\documentclass[a4paper,12pt,reqno]{amsart}
\usepackage{amssymb}
\usepackage{amsmath}
\usepackage{enumitem}
\usepackage{mathrsfs}
\usepackage[all]{xy}
\usepackage{mathtools}
\usepackage{hyperref}
\usepackage{url}

\usepackage[OT2,T1]{fontenc}
\DeclareSymbolFont{cyrletters}{OT2}{wncyr}{m}{n}
\usepackage[british]{babel}

\usepackage[margin=1.3in]{geometry}
\numberwithin{equation}{section} \numberwithin{figure}{section}

\DeclareMathOperator{\Gal}{Gal}

\DeclareMathOperator{\Hom}{Hom}

 \DeclareMathOperator{\Val}{Val}

\DeclareMathOperator{\HH}{H} \DeclareMathOperator{\cha}{char}

\DeclareMathSymbol{\Sha}{\mathalpha}{cyrletters}{"58}

\newcommand{\OO}{\mathcal{O}}

\newcommand\FF{\mathbb{F}}
\newcommand\PP{\mathbb{P}}
\newcommand\ZZ{\mathbb{Z}}
\newcommand\NN{\mathbb{N}}
\newcommand\QQ{\mathbb{Q}}
\newcommand\RR{\mathbb{R}}
\newcommand\CC{\mathbb{C}}

\newtheorem{lemma}{Lemma}

\newtheorem{theorem}[lemma]{Theorem}
\newtheorem{proposition}[lemma]{Proposition}

\theoremstyle{definition}

\newtheorem{definition}[lemma]{Definition}
\newtheorem{remark}[lemma]{Remark}

\numberwithin{lemma}{section}

\begin{document}

\title[The Hasse principle for lines]
{The Hasse principle for lines on diagonal surfaces}

\author{\sc J\"{o}rg Jahnel}
\address{J\"{o}rg Jahnel \\
Department Mathematik \\
Walter-Flex-Strasse 3\\
Universit\"{a}t Siegen \\
D-57072 \\
Siegen \\
Germany.}
\email{jahnel@mathematik.uni-siegen.de}
\urladdr{http://www.uni-math.gwdg.de/jahnel}

\author{\sc Daniel Loughran}
\address{Daniel Loughran \\
Leibniz Universit\"{a}t Hannover,
Institut f\"{u}r Algebra, Zahlentheorie
    und Diskrete Mathematik\\
Welfengarten 1\\
30167 Hannover\\
Germany.}
\email{loughran@math.uni-hannover.de}
\urladdr{http://www.iazd.uni-hannover.de/$\sim$loughran/}


%\classno{11D45 (primary), 14G05, 11N37 (secondary)}

%\extraline{The author is sponsored by an EPRSC student scholarship.}

\begin{abstract}
	Given a number field $k$ and a positive integer $d$, in this paper we consider the following
	question: does there exist a smooth diagonal surface of degree $d$ in $\PP^3$ over $k$
	which contains a line over every completion of $k$, yet no line over $k$?
	We answer the problem using Galois cohomology,
	and count the number of counter-examples using a result of Erd\H{o}s.
\end{abstract}

\maketitle

\thispagestyle{empty}

%\tableofcontents

\section{Introduction} \label{sec:intro}
A class of varieties over a number field $k$
is said to satisfy the \emph{Hasse principle} if, for each variety in the class,
the existence of a rational point over every completion of $k$ implies the existence
of a rational point over $k$.
In this paper we are interested in Hasse-type principles for linear subspaces on varieties.
Namely, given some $r \in \NN$ and a class of varieties embedded in some fixed projective space,
whether the existence of a linear subspace of dimension $r$ on each variety $X$ in the class
over every completion of $k$ implies the existence of a linear subspace of dimension $r$
on $X$ defined over $k$. Note that there is a Hilbert scheme 
(see \cite{FGA}) which parametrises the linear subspaces of dimension $r$
on each such $X$, and our question is equivalent to asking whether these schemes
satisfy the (usual) Hasse principle.

Birch \cite{Bir57} studied the existence of linear subspaces 
on complete intersections of odd degree embedded in a given projective
space~$\PP^n$ over a number field $k$. He showed that there is always a linear subspace of given dimension
$r$ over $k$, provided that $r$ is sufficiently small in terms of $n$ and the degrees 
of the defining equations (see \cite{Bra14} and \cite{Die10} for recent progress).
In particular, here the Hasse principle for such linear subspaces trivially holds.

In \cite{JL15}, we considered problems of this type for lines on del Pezzo
surfaces. We showed, for example, that the Hasse principle for lines
holds for quadric surfaces in $\PP^3$ \cite[Lem.~3.4]{JL15}. For cubic surfaces however, 
it can fail, and in \cite{JL15} we gave explicit counter-examples.

In this paper, we consider this problem for lines on diagonal surfaces, i.e.~for surfaces of the form
\begin{equation} \label{eqn:diagonal}
	a_0x_0^d + a_1x_1^d + a_2x_2^d+ a_3x_3^d = 0,
\end{equation}
where $a_i \in k^*$. Here it is well-known (see e.g.~\cite[Ex.~6.3]{BS07}) that when $d\geq 3$, 
there are exactly $3d^2$ lines over $\bar k$.
Namely, let $\zeta_{2d}$ be a primitive $2d$-th root of unity and let
$\sqrt[d]{(a_i/a_j)}$  denote a fixed choice of $d$th root of $a_i/a_j$.
Then the lines over~$\bar k$ are given by
\begin{align}
	L_1^{i,j}: \quad x_0 = \zeta_{2d}^{2i+1}\cdot \sqrt[d]{(a_1/a_0)}\cdot x_1,\quad  
	&x_2 = \zeta_{2d}^{2j+1} \cdot \sqrt[d]{(a_3/a_2)}\cdot x_3.\nonumber\\
	L_2^{i,j}: \quad x_0 = \zeta_{2d}^{2i+1} \cdot \sqrt[d]{(a_2/a_0)} \cdot x_2,\quad  
	&x_1 = \zeta_{2d}^{2j+1} \cdot \sqrt[d]{(a_3/a_1)} \cdot x_3. \label{eqn:lines} \\
	L_3^{i,j}: \quad x_0 = \zeta_{2d}^{2i+1} \cdot \sqrt[d]{(a_3/a_0)} \cdot x_3,\quad  
	&x_1 = \zeta_{2d}^{2j+1} \cdot \sqrt[d]{(a_2/a_1)}\cdot x_2. \nonumber
\end{align}
Here $(i,j)$ runs over all elements of $(\ZZ/d\ZZ)^2$ (i.e. $\zeta_{2d}^{2i+1}$ runs over all 
$2d$-th roots of unity $\zeta$ for which $\zeta^d=-1$).
One quickly sees that if the Hasse principle for lines fails, then
certain combinations of elements of $k^*$ are $d$th powers locally, but not globally.
In particular this problem is of a very similar flavour to the Grunwald-Wang theorem
(see e.g.~\cite[Thm.~9.1.3(ii)]{NSW00}),
which calculates precisely when  the Tate-Shafarevich group
$$\Sha(k,\mu_d) = \ker\left(\HH^1(k,\mu_d) \to \prod_v \HH^1(k_v,\mu_d)\right)
= \ker\left(k^*/k^{*d} \to \prod_v k_v^*/k_v^{*d}\right)$$
is trivial (for example, it is always trivial when $d$ is odd).

For del Pezzo surfaces, our results were uniform with respect to the number field
(e.g.~if we were able to construct counter-examples for a given degree over one number field, then we
could also construct counter-examples for this degree over every number field). 
For diagonal surfaces however, much like
the Grunwald-Wang theorem, the answer depends intimately on the arithmetic
of the field. Our main result states that for odd degree, the answer is completely controlled by cohomology,
whereas for even degree there are extra subtleties depending on whether
or not $(-1)$ is a $d$th power in $k$.

\begin{theorem} \label{thm:cohomology}
	Let $d \geq 3$ and let $k$ be a number field.
	If $d$ is odd, then every diagonal surface of degree $d$ over $k$ satisfies
	the Hasse principle for lines if and only if
	$$\HH^1(k(\mu_d)/k,\mu_d) = 0.$$
	If $d$ is even, then every diagonal surface of degree $d$ over $k$ satisfies
	the Hasse principle for lines if and only if
	$$\HH^1(k(\mu_d)/k,\mu_d) = 0 \quad \mbox{and} \quad (-1) \not \in k^{*d}.$$
\end{theorem}
Note that if we write $d=2^ne$ where $n \geq 0$ and  $e$ is odd, then $(-1) \in k^{*d}$
if and only if $\mu_{2^{n+1}} \subset k$.
The cohomology group in Theorem \ref{thm:cohomology} naturally arises for us via inflation-restriction, namely
we have
\begin{equation} \label{eqn:H^1}
	\HH^1(K/k,\mu_d(K))=\ker\left(k^*/k^{*d} \to K^* / K^{*d}\right),
\end{equation}
for any finite Galois field extension $k \subset K$.
We calculate the cohomology group appearing in Theorem \ref{thm:cohomology}
in \S \ref{sec:cohomology}, from which we are able to obtain the following results.

\begin{theorem} \label{thm:odd_hold}
	Let $d \geq 3$ be odd and let $k$ be a number field.
	Suppose that one of the following conditions holds.
	\begin{enumerate}
		\item $\mu_d(k) = \{1\}$.
		\item $\mu_d \subset k$.
		\item $q \not \equiv 1 \bmod p$ for all primes $p,q \mid d$.
	\end{enumerate}
	Then diagonal surfaces of degree $d$ over $k$ satisfy the Hasse principle
	for lines.
\end{theorem}

Condition $(1)$ holds for example if $k=\QQ$, condition $(2)$ holds when $k~=~\QQ(\mu_d)$,
and condition $(3)$ holds if $d$ is an odd prime power or $d=15$.
In the original Grunwald-Wang theorem,
one quickly reduces to the case where the exponent is a prime power.
Here however this cannot be done, as the prime divisors of $d$ can interact
in a non-trivial way. In particular, contrary to the case of Grunwald-Wang, the Hasse principle for lines
can fail for diagonal surfaces of odd degree. 

\begin{theorem} \label{thm:odd_fail}
	Let $d \geq 3$ be odd.
	Suppose that there exist primes $p,q \mid d$,
	such that $q \equiv 1 \bmod p$. Then there exist a number field $k$ 
	and a diagonal surface of degree $d$ over $k$ which fails
	the Hasse principle for lines.
\end{theorem}

Theorem \ref{thm:odd_fail} applies for example when $d=21$.
Counting the number of $d$ to which Theorem \ref{thm:odd_fail} applies
is an interesting problem in its own right in analytic number theory.
Namely, let
$$D(x) = \# \{ d \leq x : q \not \equiv 1 \bmod p \textrm{ for all primes } p,q \mid d\}.$$
Let $D_{\mathrm{sf}}(x)$ denote the corresponding counting function where $d$ is also assumed
to be square-free. Then Erd\H{o}s \cite{Erd48} (see also \cite[Thm.~11.23]{MV07}) has shown that 
$$D_{\mathrm{sf}}(x) \sim e^{-\gamma} \frac{ x}{\log \log \log x}, \quad \textrm{ as } x \to \infty.$$
A minor modification of his proof yields the upper bound
$$D(x) \ll \frac{ x}{\log \log \log x}.$$
In particular, Theorem \ref{thm:odd_fail} applies to $100\%$ of all odd integers $d$.

Using Theorem \ref{thm:cohomology} it is easy to see that for each even $d\geq 4$, 
there exists a number field $k$ for which every diagonal surface of degree $d$ over $k$ satisfies
the Hasse principle for lines (e.g., take $k = \QQ(\mu_d)$).
However in contrast to part $(3)$ of Theorem~\ref{thm:odd_hold},
the Hasse principle for lines can fail in every even degree.

\begin{theorem} \label{thm:even_fail}
	Let $d\geq 4$ be even and let $k$ be a totally real number field.
	Then there exists a diagonal surface of degree $d$ over $k$ which
	fails the Hasse principle for lines.
\end{theorem}
In particular, Theorem \ref{thm:even_fail} applies when $k=\QQ$.

%\newpage

\subsection{Examples}
\subsubsection{Even degree}
The counter-examples which we construct in this paper are reasonably explicit.
For example, let $p \equiv 1 \bmod 8$ be a prime. We claim that the surface
\begin{equation} \label{eqn:surface}
	S: \quad x_0^4 + 2^2x_1^4 - p^2 x_2^4 - (2p)^2x_3^4 = 0
\end{equation}
over $\QQ$ fails the Hasse principle for lines.
Let us give an elementary argument why this is the case. Given a field extension $\QQ \subset k$, the surface $S$ contains a line
over $k$ if and only if
\begin{equation} \label{eqn:explicit}
	{-2}^2 \in k^{*4} \quad \mbox{or} \quad p^2 \in k^{*4} \quad \mbox{or} 
	\quad (2p)^2 \in k^{*4}.
\end{equation}
This is most easily seen using the explicit description of the Hilbert scheme
of lines given in \S \ref{sec:Hilbert}. In particular, there is no line
defined over $\QQ$. Clearly there is a line defined over $\RR$, 
and there is a line over $\QQ_2$ since $p \in \QQ_2^{*2}$.
So let $\ell$ be an odd  prime. If $\ell \equiv 1 \bmod 4$, then 
$(-1) \in \QQ_{\ell}^{*2}$ and
hence $-2^2=(1+i)^4$ is a fourth power in $\QQ_\ell$, where $i^2 = -1$. 
If $\ell \equiv 3 \bmod 4$, then every square in $\FF_{\!\ell}^*$ is also a fourth power.
Whence it follows from Hensel's lemma and \eqref{eqn:explicit}
that there is a line defined over $\QQ_\ell$. Thus $S$ fails
the Hasse principle for lines, as claimed.

Here the appearance of $2^2$ in the coefficients of \eqref{eqn:surface}
is very important for the construction. It appears naturally
through the consideration of the cohomology group 
$\HH^1(\QQ(i)/\QQ, \mu_4)$, as one would expect from Theorem~\ref{thm:cohomology}.
This cohomology group is isomorphic
to $\ZZ/2\ZZ$ (see Lemma \ref{lem:GWW}). A choice of representative for the non-trivial
element is given by $-2^2$, as $-2^2=(1+i)^4$ is not a fourth power in $\QQ$ but is a fourth power in $\QQ(i)$
(see \eqref{eqn:H^1}). This fact  naturally arose in the above argument,
as it implies that $-2^2 \in \QQ_\ell^{*4}$ for those primes $\ell$ which split in
$\QQ(i)$ (i.e. those primes $\ell \equiv 1 \bmod 4$).
This construction of counter-examples is generalised in Lemma~\ref{lem:b} below.
If $(-1) \in k^{*d}$ then there are counter-examples of a different kind,
which are moreover easier to construct; see Lemma \ref{lem:-1} below.

\subsubsection{Odd degree}
The counter-examples in odd degree are of a similar flavour, though here it is slightly
trickier to find conditions which guarantee that the group $\HH^1(k(\mu_d)/k, \mu_d)$
is non-trivial
(this group was non-trivial above for special reasons, namely as we were in the \emph{special case}
of Definition~\ref{def:special}).

As an example, let $p$ and $q$ be two odd primes with $q \equiv 1 \bmod p$. Put $d=pq$
and let $k=\QQ(\mu_p)$. Then we claim that $\HH^1(k(\mu_d)/k, \mu_d) \neq 0$. To see
this, first note that $\Gal(k(\mu_d)/k) \cong \ZZ/(q-1)\ZZ$. As $q \equiv 1 \bmod p$,
there exists an intermediate field extension $k \subset L \subset k(\mu_d)$
with $\Gal(L/k) \cong \ZZ/p\ZZ$. By Kummer theory, there exists some $\alpha \in k^*$
such that $L=k(\sqrt[p]{\alpha})$. Hence $\alpha$ is not a $p$th power
in $k$, but is a $p$th power in $L$. A simple argument shows that $\alpha^q$ is not
a $d$th power in $k$, thus by \eqref{eqn:H^1} we see that $\alpha^q$ 
gives rise to a non-trivial element of $\HH^1(k(\mu_d)/k, \mu_d)$, as required.
As for the counter-examples, let $\beta \in k^*$ and consider the surface
$$
	S: \quad x_0^{pq} - \alpha^{q} x_1^{pq} - \beta^p x_2^{pq} + \alpha^q \beta^p x_3^{pq} = 0.
$$
A similar argument to the case of even degree shows that $S$ contains a line over all
but finitely many completions of $k$. One needs to choose $\beta$ carefully
to ensure that there are lines \emph{everywhere} locally, whilst there still being no line
globally.

As an explicit example, let $p = 3$ and $q = 7$, so that $L =\QQ(\zeta_3, \zeta_7 + \zeta_7^{-1})$
where $\zeta_3$ and $\zeta_7$ are choices of primitive $3$rd and $7$th roots of
unity, respectively. Take
$$\rho = (\zeta_{7} + \zeta_{7}^{-1}) + \zeta_{3}(\zeta_{7}^2 + \zeta_{7}^{-2})
+ \zeta_{3}^{2}(\zeta_{7}^3 + \zeta_{7}^{-3}) \in L.$$
A simple calculation shows that if $\sigma \in \Gal(L/k)$ is a non-trivial element,
then $\sigma(\rho)/\rho$ is a non-trivial third root of unity.
In particular $\rho \not \in k$. Moreover, this shows that $\alpha = \rho^3 = 14 + 21\zeta_3  \in k$ is a suitable
choice for~$\alpha$, as above. In order to obtain the required counter-example, one may take $\beta$
to be a rational prime which is both $1 \bmod 3$ and $1 \bmod 7^2$ (e.g. take $\beta = 883$).
See the proof of Lemma \ref{lem:a} for details of why this choice works.

\subsection*{Notation}
Let $d \in \NN$ and let $k$ be a field whose characteristic is coprime to $d$. We denote by $\mu_d$ the group scheme
of $d$th roots of unity over $k$. By abuse of notation, we shall say that $\mu_d \subset k$
if $k$ contains a primitive $d$th root of unity, and denote by $k(\mu_d)$ the field
given by adjoining all $d$th roots of unity. We have $[k(\mu_d) : k] \mid \varphi(d)$, where $\varphi$
is Euler's phi function.

If $a \in k^*$, then the notation $k( a^{1/d} )$ is meaningless
in general; we shall only use it if $\mu_d \subset k$, in which case $k( a^{1/d})$ is defined
to be the splitting field of the polynomial $x^d - a$.

Let $k \subset K$ be a Galois field extension and $G$ a finite \'{e}tale commutative group scheme over $k$.
Then for $i \geq 0$, we use the notation
$\HH^i(K / k , G(K))$
for the Galois cohomology group
$\HH^i(\Gal(K/k), G(K) ).$
If $G$ is split over $K$, then we instead write
$\HH^i(K / k , G).$

\subsection*{Acknowledgements}
The authors would like to thank Matthias Sch\"{u}tt and Efthymios Sofos for useful comments, 
and the anonymous referee for their careful reading of our paper.

\section{Generalities}

\subsection{Cohomological calculations} \label{sec:cohomology}
We begin by calculating the cohomology group which arises in Theorem \ref{thm:cohomology}.
We will assume that the reader is familiar with the basic tools of Galois cohomology, as
can be found in \cite{NSW00} or \cite{Ser97}.

\begin{definition} \label{def:special}
	Let $p^n$ be a prime power and let $k$ be a field with $\cha(k)=0$.
	Then we shall say that $(k,p^n)$ is the \emph{special case} if 
	$$p = 2,\,\, n \geq 2 \quad \mbox{ and} \quad k \cap \QQ(\mu_{2^n}) \mbox{ is totally real}.$$
\end{definition}

This arises in the following lemma, which is usually 
shown during the course of the proof of the Grunwald-Wang theorem.

\begin{lemma} \label{lem:GWW}
	Let $p^n$ be a power prime and let $k$ be a field with $\cha(k) = 0$.
	If $(k,p^n)$ is not the special case, then
	$$
		\HH^i(k(\mu_{p^n})/k,\mu_{p^n})=0 \quad \mbox{for all } i \geq 1.
	$$
	Otherwise
	$$
		\HH^i(k(\mu_{2^n})/k,\mu_{2^n}) \cong \ZZ/2\ZZ \quad \mbox{for all } i \geq 1,
	$$
\end{lemma}
\begin{proof}
	 See \cite[Prop.~9.1.4]{NSW00}.
\end{proof}

\begin{proposition} \label{prop:cohomology}
	Let $d \in \NN$ and let $k$ be a field with $\cha(k) = 0$.
	Let $k \subset K$ be a finite abelian extension with $\mu_d \subset K$.
	Then
	\begin{enumerate}
	\item[$(a)$]
	$\HH^1(K/k, \mu_d)  \cong \prod_{p^n \mid\mid d} \HH^1(K/k, \mu_{p^n}).$
	\end{enumerate}
	Let $p^n \mid \mid d$. If $(k,p^n)$ is not the special case, then
	\begin{enumerate}
	\item[$(b)$]
	$\HH^1(K/k, \mu_{p^n}) \cong \Hom(\Gal(K/k(\mu_{p^n})), \mu_{p^n}(k)).$
	\end{enumerate}
\end{proposition}
\begin{proof}
	Part $(a)$ follows simply from the fact that if $p^n \mid \mid d$,
	then $\mu_d = \mu_{p^n} \oplus \mu_{d/p^n}$. For part $(b)$, 
	by inflation-restriction we have the exact sequence
	\begin{align*}
		0 &\longrightarrow \HH^1(k(\mu_{p^n})/k,\mu_{p^n}) \longrightarrow \HH^1(K/k,\mu_{p^n}) \\
		&\longrightarrow \HH^1(K/k(\mu_{p^n}),\mu_{p^n})^{\Gal(k(\mu_{p^n})/k)} \longrightarrow \HH^2(k(\mu_{p^n})/k,\mu_{p^n}).
	\end{align*}
	Thus by Lemma \ref{lem:GWW} we find that
	\begin{equation} \label{eqn:bad}
		\HH^1(K/k,\mu_{p^n}) \cong \HH^1(K/k(\mu_{p^n}),\mu_{p^n})^{\Gal(k(\mu_{p^n})/k)}.
	\end{equation}
	As $\Gal(K/k(\mu_{p^n}))$ acts trivially on	$\mu_{p^n}$, we see that \eqref{eqn:bad}
	is isomorphic to the group $\Hom(\Gal(K/k(\mu_{p^n})), \mu_{p^n})^{\Gal(k(\mu_{p^n})/k)}$.
	Part $(b)$ now follows from the fact that, as the extension $K/k$ is abelian, 
	the natural conjugation action of the group $\Gal(k(\mu_{p^n})/k)$ on $\Gal(K/k(\mu_{p^n}))$ is trivial.
\end{proof}

\begin{lemma} \label{lem:non-trivial}
	Let $p^n$ be a prime power and let $k$ be a field with $\cha(k) = 0$.
	Let $K/k$ a finite abelian extension 
	with $\mu_{p^n} \subset K$.
	\begin{enumerate}
		\item[$(a)$] If $(k,p^n)$ is not the special case, then 
			$\HH^1(K/k, \mu_{p^n})$ is non-trivial if and only if
			$$ \mu_p \subset k \quad \mbox{and} \quad p \mid [K:k(\mu_{p^n})].$$
		\item[$(b)$] If $p=2$ and $(k,2^n)$ is the special case,
			then $\HH^1(K/k, \mu_{2^n})$ is non-trivial.
	\end{enumerate}
\end{lemma}
\begin{proof}
	Part $(a)$ follows immediately from Propositon \ref{prop:cohomology}.
	Part $(b)$ follows from Lemma \ref{lem:GWW} and the injectivity of the inflation map
	\begin{equation*}
		\HH^1(k(\mu_{2^n})/k, \mu_{2^n}) \longrightarrow \HH^1(K/k, \mu_{2^n}). \qedhere
	\end{equation*} 
\end{proof}

\begin{lemma} \label{lem:pq}
	Let $d \in \NN$ and let $k$ be a field with $\cha(k) = 0$.
	Assume that 
	$$\HH^1(k(\mu_d)/k, \mu_d) \neq 0.$$
	Then either
	\begin{enumerate}
		\item[$(a)$] There exist distinct primes $p,q \mid d$ with $p^n \mid\mid d$
		for some $n \in \NN$ such that
		$$\HH^1( k(\mu_q)/k, \mu_{p^n}(k(\mu_q))) \neq 0,$$
	\end{enumerate}
	or
	\begin{enumerate}
		\item[$(b)$] $2^n \mid \mid d$ for some $n \geq 2$ and $(k,2^n)$ is the special case,
		i.e.
		$$\HH^1( k(\mu_{2^n})/k, \mu_{2^n}) \neq 0.$$
	\end{enumerate}
\end{lemma}
\begin{proof}
	By Proposition \ref{prop:cohomology}$(a)$, there exists some
	prime $p$ with $p^n \mid\mid d$	such that 
	$$ \HH^1(k(\mu_d)/k, \mu_{p^n}) \neq 0.$$
	If $(k,p^n)$ is the special case, then
	the result holds by Lemma~\ref{lem:GWW}.
	Otherwise, Lemma \ref{lem:non-trivial} implies that
	\begin{equation} \label{eqn:p1}
		\mu_p \subset k \quad \mbox{and} \quad p \mid [k(\mu_d):k(\mu_{p^n})].
	\end{equation}
	We claim that in this situation there exists some prime $q \mid d/p^n$,
	such that
	\begin{equation} \label{eqn:p2}
		p \mid [k(\mu_q) : k(\mu_{p^n}) \cap k(\mu_q) ].
	\end{equation}
	Indeed, consider a factorisation $d = p^nq_1^{n_1}\cdots q_r^{n_r}$ of $d$
	and the associated tower of fields
	$$k(\mu_{p^n}) \subset k(\mu_{p^{n} q_1^{n_1}}) \subset \cdots \subset k(\mu_{d}).$$
	By \eqref{eqn:p1} we obtain
	$$p \mid [k(\mu_{p^{n} q_1^{n_1} \cdots q_i^{n_{i}}}): k(\mu_{p^{n} q_1^{n_1} \cdots q_{i-1}^{n_{i-1}}})]$$
	for some $i$. Hence
	$$p \mid [k(\mu_{p^{n} q_i^{n_{i}}}): k(\mu_{p^n})].$$
	%Use theorem 2.6 in K. Conrad's notes on Galois correspondence at work.
	Thus we have
	\begin{equation} \label{eqn:q}
		p \mid [k(\mu_{p^{n} q_i}): k(\mu_{p^n})],
	\end{equation}
	since $[k(\mu_{p^{n} q_i^{n_{i}}}): k(\mu_{p^n q_i})]$ is a power of $q_i$.
	From this we deduce \eqref{eqn:p2}, with $q = q_i$.
	
	Next choose the largest $m \leq n$ such that $\mu_{p^m} \subset k(\mu_q)$.
	Note that $m \geq 1$ by \eqref{eqn:p1}. Then
	$\mu_{p^n}(k(\mu_q)) = \mu_{p^m}$, and hence
	$$\HH^1( k(\mu_q)/k, \mu_{p^n}(k(\mu_q))) = \HH^1( k(\mu_q)/k, \mu_{p^m}).$$
	As $k(\mu_{p^m}) \subset k(\mu_{p^n}) \cap k(\mu_q)$,
	it follows from \eqref{eqn:p2} that
	$$p \mid [k(\mu_q) : k(\mu_{p^m})].$$
	Hence by Lemma \ref{lem:non-trivial}, we obtain the result.
\end{proof}

\subsection{The Hilbert scheme of lines} \label{sec:Hilbert}

Let $d \geq 3$ and let $S$ be a smooth diagonal surface of the form
$$
	a_0x_0^d + a_1x_1^d + a_2x_2^d+ a_3x_3^d = 0,
$$
over a field $k$ whose characteristic is coprime to $d$.
A simple calculation using the explicit description of the lines \eqref{eqn:lines} shows that the Hilbert scheme of 
lines of $S$ is isomorphic to a disjoint union $\mathcal{L} = \mathcal{L}_1 \sqcup \mathcal{L}_2 \sqcup \mathcal{L}_3$,
where
\begin{align} \label{eqn:general}
	\mathcal{L}_1 &: \{ x^d = -a_1/a_0 \} \times \{ y^d = - a_2/a_3 \}, \nonumber \\
	\mathcal{L}_2 &: \{ x^d = -a_2/a_0 \} \times \{ y^d = - a_3/a_1 \}, \\
	\mathcal{L}_3 &: \{ x^d = -a_3/a_0 \} \times \{ y^d = - a_1/a_2 \}. \nonumber 
\end{align}
Here each $\mathcal{L}_i$ is viewed as a subscheme of $\mathbb{A}^2$,
and is finite \'{e}tale of degree $d^2$ over~$k$. We keep this notation throughout the paper.
A case of special interest for us occurs for surfaces of the shape
$$
	x_0^d + ax_1^d + bx_2^d+ abx_3^d = 0,
$$
where we have 
\begin{align} \label{eqn:special}
	\mathcal{L}_1 &: \{ x^d = -a \} \times \{ y^d = - a \}, \nonumber \\
	\mathcal{L}_2 &: \{ x^d = -b \} \times \{ y^d = - b \}, \\
	\mathcal{L}_3 &: \{ x^d = -ab \} \times \{ y^d = - a/b \}. \nonumber 
\end{align}

\section{Positive results}
The aim of this section is to prove that the Hasse principle for lines
holds in the cases stated in Theorem \ref{thm:cohomology}.
\subsection{The Hasse principle for finite \'{e}tale schemes}
We begin with some general remarks on the Hasse principle for 
finite \'{e}tale schemes, following our exposition given in \cite[\S2.1]{JL15}.

Let $X$ be a finite \'{e}tale scheme over a field $k$.
Associated to this is a splitting field $K/k$, whose Galois group $\Gamma$
acts faithfully on the set $X(K)$ in a natural way. The following lemma
gives a group theoretic criterion for the Hasse principle to fail 
when $k$ is a number field.

\begin{lemma}\label{lem:Chebotarev}
	Let $X$ be a finite \'{e}tale scheme over a number field $k$. Let $K/k$
	denote the splitting field, with Galois group $\Gamma$. Then $X$ is locally soluble
	at all but finitely many places of $k$, but not soluble over $k$,
	if and only if on the associated
	$\Gamma$-set $X(K)$ each conjugacy class of $\Gamma$ acts with a fixed point,
	but the group $\Gamma$ acts without a fixed point.
	
	In which case, for each place $v$ of $k$ which is either archimedean
    or unramified in~$K$, the scheme $X$ admits a $k_v$-point.
\end{lemma}
\begin{proof}
	This is a simple application of the Chebotarev density theorem;
	see \cite[Lem.~2.2]{JL15}.
\end{proof}

\subsection{The Hasse principle for lines}
We now show that the Hasse principle for lines
holds in the cases stated in Theorem \ref{thm:cohomology}.
We begin with the case where $\mu_d \subset k$.

\begin{proposition} \label{prop:mu_d}
	Let $d \geq 3$ and let $k$ be a number field with $\mu_d \subset k$.
	Suppose that $(-1) \not \in k^{*d}$ if $d$ is even.
	Then the Hasse principle for lines holds for diagonal surfaces of degree $d$ over $k$.
\end{proposition}
\begin{proof}
	Assume that there exists a diagonal surface $S$ of degree $d$ over $k$ which fails the Hasse
	principle for lines. Let $\mathcal{L}$ be its Hilbert scheme of lines,
	with splitting field $K$ and Galois group $\Gamma$.
	Using the explicit description \eqref{eqn:general} and the fact that $\mu_d \subset k$, one sees that $\Gamma$ is abelian. Thus the 
	stabiliser of each point $\ell \in \mathcal{L}_i(K)$
	depends only on $i$; denote the corresponding subgroup by $\Gamma_i$.
	Note that 
	\begin{equation} \label{eqn:intersection_trivial}
		\Gamma_1 \cap \Gamma_2 \cap \Gamma_3 = 0
	\end{equation}
	as the action of $\Gamma$ on $\mathcal{L}(K)$ is faithful.
	By Lemma \ref{lem:Chebotarev} we have 
	\begin{equation} \label{eqn:Chebotarev}
		\Gamma = \Gamma_1 \cup \Gamma_2 \cup \Gamma_3 \quad \mbox{and} \quad \Gamma_i \neq \Gamma \quad \mbox{for } i=1,2,3.
	\end{equation}	  
	We claim that 	
	\begin{equation} \label{eqn:i<>j}
		\Gamma_i \cap \Gamma_j = 0 \quad \mbox{for all } i \neq j.
	\end{equation}
	To prove this, without loss of generality we may take $(i,j) = (1,2)$.
	Suppose that there exists some non-zero $\gamma \in \Gamma_1 \cap \Gamma_2$.
	As a group can never be a union of two proper subgroups,
	we see that there exists some $\gamma' \in \Gamma \setminus (\Gamma_1 \cup \Gamma_2) \subset \Gamma_3$.
	To prove \eqref{eqn:i<>j}, by \eqref{eqn:Chebotarev} it suffices to show that
	$\gamma + \gamma' \not \in \Gamma_i$ for each $i$. To see this, note that $\gamma$
	stabilises each point of $\mathcal{L}_1(K)$ and $\mathcal{L}_2(K)$,
	whereas $\gamma'$ does not. Similarly, $\gamma'$ stabilises each point
	of $\mathcal{L}_3(K)$, whereas by \eqref{eqn:intersection_trivial}
	we see that  $\gamma$ does not. This proves \eqref{eqn:i<>j}.
	
	Next, by \eqref{eqn:Chebotarev} and inclusion-exclusion we have 
	$$
		\frac{1}{[\Gamma:\Gamma_1]} + \frac{1}{[\Gamma:\Gamma_2]}+ \frac{1}{[\Gamma:\Gamma_3]} > 1.
	$$
	Hence without loss of generality
	$$[\Gamma:\Gamma_1] = 2 \quad \mbox{and} \quad [\Gamma:\Gamma_2] \leq 3.$$
	Moreover, by \eqref{eqn:i<>j}, the natural map
	$$ \Gamma \to \Gamma/\Gamma_1 \times \Gamma/\Gamma_2$$
	is an embedding, so $\# \Gamma \leq 6$.
	However a cyclic group cannot satisfy Lemma \ref{lem:Chebotarev}, hence
	\begin{equation} \label{eqn:Z/2ZxZ/2Z}
		\Gamma \cong \ZZ/2\ZZ \times \ZZ/2\ZZ.
	\end{equation}
	
	To proceed, we need to explicitly construct the generic 
	Galois group together with its action
	on the Hilbert scheme of lines.
	Consider a surface of the form 
	\begin{equation} \label{eqn:generic}
		x_0^d + a_1x_1^d+ a_2x_2^d+ a_3x_3^d = 0
	\end{equation}		
	over the function field $k(a_1,a_2,a_3)$. 
	If $d$ is odd, then the splitting field of \eqref{eqn:generic} is
	$$k(\sqrt[d]{a_1}, \sqrt[d]{a_2}, \sqrt[d]{a_3}),$$
	whose Galois group is isomorphic to $(\ZZ/d\ZZ)^3$.
	Thus we see	that $\Gamma$ has odd order, which contradicts \eqref{eqn:Z/2ZxZ/2Z}.
	Hence the Hasse principle for lines holds when $d$ is odd.
	Now suppose that $d$ is even. The splitting field of \eqref{eqn:generic} is
	$$k(\mu_{2d}, \sqrt[d]{a_1}, \sqrt[d]{a_2}, \sqrt[d]{a_3}),$$
	whose Galois group $G$ is isomorphic to $\ZZ/2\ZZ \times (\ZZ/d\ZZ)^3.$
	It is generated by the elements $\sigma_0,\sigma_1,\sigma_2,\sigma_3$, where
	$$
	\begin{array}{ll}
		\sigma_0(\zeta_{2d}) = - \zeta_{2d},\quad  &\sigma_0(\sqrt[d]{a_i}) =\sqrt[d]{a_i}, \quad  i=1,2,3,\\
		\sigma_i(\zeta_{2d}) = \zeta_{2d},\quad &\sigma_i(\sqrt[d]{a_i}) = \zeta_{2d}^2 \sqrt[d]{a_i}, \quad
		\sigma_i(\sqrt[d]{a_j}) = \sqrt[d]{a_j}, \quad i=1,2,3,\, j \neq i,
	\end{array}
	$$
	and $\zeta_{2d}$ is a fixed choice of primitive $2d$-th root of unity.
	The action of $G$ on the Hilbert scheme of lines of \eqref{eqn:generic} (see \eqref{eqn:general})
	gives rise to a faithful representation
	\begin{align*}
		\rho: G &\longrightarrow (\ZZ/d\ZZ)^2 \times (\ZZ/d\ZZ)^2 \times (\ZZ/d\ZZ)^2 \\
		\sigma_0 &\mapsto ((d/2,d/2), (d/2,d/2), (d/2,d/2)) \\
		\sigma_1 &\mapsto ((1,0), (0,-1), (0,1)) \\
		\sigma_2 &\mapsto ((0,1), (1,0), (0,-1)) \\
		\sigma_3 &\mapsto ((0,-1), (0,1), (1,0)).
	\end{align*}
	We identify $G$ with its image under this map, viewed as an additive subgroup
	of a free $(\ZZ/d\ZZ)^2$-module of rank $3$.
	With this notation, an element of $G$ fixes a line if and only if one of its
	coordinates is $(0,0)$, and an element of $G$
	fixes $\mu_{2d}$ if and only if it lies in the subgroup $\langle \sigma_1,\sigma_2,\sigma_3 \rangle.$
	
	We now return to the group $\Gamma$. This may be
	identified with a subgroup of $G$, acting on the lines of $S$ via the above representation.	
	By \eqref{eqn:i<>j} and \eqref{eqn:Z/2ZxZ/2Z}, we see that $\Gamma$ consists of elements
	of the form $(0,0,0), (t_1,t_2,0), (t_1,0,t_3)$, and $(0,t_2,t_3)$,
	where $t_1, t_2, t_3$ are $2$-torsion elements in $(\ZZ/d\ZZ)^2$.
	However the $2$-torsion of the generic Galois group $G$ is generated by
   	\begin{align*}
   		&((d/2,d/2), (d/2,d/2), (d/2,d/2)), \\
   		&((d/2,0), (0,d/2), (0,d/2)), \\
 	   	&((0,d/2), (d/2,0), (0,d/2)), \\
   		&((0,d/2), (0,d/2), (d/2,0)).
	\end{align*}
	The only $2$-torsion elements with exactly one coordinate $(0,0)$ are
	\begin{align*}
   		&\gamma_1 = ((0,0), (d/2,d/2), (d/2,d/2)), \\
		&\gamma_2 = ((d/2,d/2), (0,0), (d/2,d/2)), \\
   		&\gamma_3 = ((d/2,d/2), (d/2,d/2), (0,0)),
   	\end{align*}
   	and hence $\Gamma = \{0, \gamma_1,\gamma_2,\gamma_3\}$.
	However each $\gamma_i$ lies in $\langle \sigma_1,\sigma_2,\sigma_3 \rangle$,
	so fixes $\mu_{2d}$. Thus $\mu_{2d} \subset k$, which contradicts 
	the assumption that $(-1) \not \in k^{*d}$. 
	Therefore the Hasse principle for lines holds when $d$ is even.
\end{proof}

\begin{proposition}
	Let $d \geq 3$ and let $k$ be a number field such that
	$$\HH^1(k(\mu_d)/k, \mu_d) = 0.$$
	Suppose moreover that 
	$(-1) \not \in k^{*d}$
	if $d$ is even.
	Then the Hasse principle for lines holds for diagonal surfaces of degree $d$ over $k$.
\end{proposition}
\begin{proof}
	Assume that there exists a diagonal surface $S$ 
	of degree $d$ over $k$ which fails the Hasse
	principle for lines. Then $S_{k(\mu_d)}$ contains a line
	everywhere locally. Note that if $d$ is even, then the conditions
	of the proposition and \eqref{eqn:H^1} imply that $(-1) \not \in k(\mu_d)^{*d}$.
	Hence by Proposition \ref{prop:mu_d} 
	we see that $S_{k(\mu_d)}$
	contains a line. Without loss of generality we have 
	$\mathcal{L}_1(k(\mu_d)) \neq \emptyset$, thus $a_1 / a_0$ 
	and $a_3/a_2$ are both $d$th powers in $k(\mu_d)$ (see \eqref{eqn:general}).
	However since $\HH^1(k(\mu_d)/k,\mu_d)=0$,
	by \eqref{eqn:H^1} we see that $a_1 / a_0$ and $a_3/a_2$ are both $d$th powers in $k$.
	Thus $S$ contains a line over~$k$, which is a contradiction.
\end{proof}

This completes the proof of the if part of Theorem \ref{thm:cohomology}.

\subsection{Proof of Theorem \ref{thm:odd_hold}}
By Theorem \ref{thm:cohomology}, it suffices to show that we have
$\HH^1(k(\mu_d)/k,\mu_d)=0$ in each of the given cases.
\begin{enumerate}
	\item If $\mu_d(k) = \{1\}$, then the group is trivial by Lemma~\ref{lem:non-trivial}$(a)$.
	\item If $\mu_d \subset k$ then it is clearly trivial.
	\item If $q \not \equiv 1 \bmod p$ for all primes $p,q \mid d$, then
		$p \nmid [ k(\mu_d) : k(\mu_{p^n})]$ for all $p^n \mid \mid d$
		(since $[ k(\mu_d) : k(\mu_{p^n})] \mid \varphi(d/p^n)]$). 
		Hence the group is trivial here by Lemma~\ref{lem:non-trivial}$(a)$. \qed
\end{enumerate}

\section{Existence of counter-examples}
The aim of this section is to construct explicit counter-examples,
which show that the conditions in Theorem \ref{thm:cohomology} are necessary.
\subsection{An existence result}
We begin with a technical lemma on the existence
of certain integers, whose primary purpose is to take care of 
the existence of a line at the bad primes.

\begin{lemma} \label{lem:beta}
	Let $k$ be a number field and let $p, q \in \NN$.
	Let $\alpha \in \OO_k$ be non-zero.
	Then there exists $\beta \in \ZZ$ such that the following
	conditions hold.
	\begin{enumerate}
	\item $\beta$ is coprime to $\alpha$.
	\item $k$ has a prime ideal $\mathfrak{b}$ such that $v_\mathfrak{b}(\beta)=1$.
	\item If $\mathfrak{p} \mid \alpha q$, then $\beta$ is a $q$th power in $k_{\mathfrak{p}}$.
	\item If $\mathfrak{p} \mid \beta$, then $\alpha$ is a $p$th power in $k_\mathfrak{p}$.	 
	\end{enumerate}
\end{lemma}
\begin{proof}
	Note that we do not assume that $p$ and $q$ are coprime.
	Denote by $\mathfrak{a}_1,\ldots, \mathfrak{a}_r$ the prime ideals
	dividing $\alpha$ with corresponding rational primes $a_1,\ldots,a_r$,
	and let $q_1,\ldots, q_s$ be the rational primes dividing $q$.
	Let $n$ and $m$ be large positive integers 
	(these are chosen so that Hensel's lemma applies below).
	We then choose~$\beta$ to be a rational prime which is completely
	split in the field extension
	$$\QQ \subset k(\mu_p,\alpha^{1/p}, \mu_{a_1^{n}},\ldots, \mu_{a_r^{n}},
	\mu_{q_1^{m}},\ldots,\mu_{q_s^{m}} ).$$
	Such primes exist by the Chebotarev density theorem, and we may moreover assume
	that $(1)$ holds. Also $(2)$ holds as $\beta$ is completely split in $k$.
	
	Next, as $\beta$ is completely split in $\smash{\QQ(\mu_{a_i^{n}})}$,
	we see that $\beta \equiv 1 \bmod a_i^{n}$ for each $i =1,\ldots,r$. Hence 
	$\beta$ is a $q$-th power in $\QQ_{a_i} \subset k_{\mathfrak{a}_i}$ by Hensel's lemma.
	A similar argument applies to the $q_i$, thus $(3)$ holds.
	
	Finally let $\mathfrak{p} \mid \beta$. Then $\mathfrak{p}$
	is completely split in $k(\mu_p,\alpha^{1/p})$, hence $\alpha$ is a $p$-th
	power in $k_{\mathfrak{p}}$, thus $(4)$ holds. This completes the proof.
\end{proof}

\subsection{Existence of counter-examples}
We now construct counter-examples as required for Theorem \ref{thm:cohomology},
starting with the easiest case where $d$ is even and $(-1) \in k^{*d}.$

\begin{lemma} \label{lem:-1}
	Let $d \geq 4$ be even and let $k$ be a  number field with
	$$(-1) \in k^{*d}.$$
	Then there exists a diagonal surface of degree $d$ over $k$ which fails the Hasse principle for lines.
\end{lemma}
\begin{proof}
	Let $\alpha$ be a rational prime which is completely split in $k$. By Lemma \ref{lem:beta},
	there exists $\beta \in \ZZ$ such that
	\begin{enumerate}
	\item $\beta$ is coprime to $\alpha$.
	\item $k$ has a prime ideal $\mathfrak{b}$ such that $v_\mathfrak{b}(\beta)=1$.
	\item If $\mathfrak{p} \mid 2\alpha$, then $\beta$ is a square in $k_{\mathfrak{p}}$.
	\item If $\mathfrak{p} \mid \beta$, then $\alpha$ is a square in $k_\mathfrak{p}$.	 	\end{enumerate}
	Consider the diagonal surface
	\begin{equation} \label{eqn:biquadratic}
		x_0^d + \alpha^{d/2} x_1^d + \beta^{d/2} x_2^d + (\alpha\beta)^{d/2} x_3^d = 0
	\end{equation}
	over $k$. Let $k \subset K$ be a field extension. As $(-1) \in k^{*d}$, using \eqref{eqn:special}
	one finds that there is a line defined $K$	if and only if
	\begin{equation} \label{eqn:biquadratic_lines}
	{\alpha^{d/2}} \in K^{*d} \quad \mbox{or} \quad {\beta^{d/2}} \in K^{*d} 
	\quad \mbox{or} \quad  (\alpha\beta)^{d/2} \in K^{*d}.
	\end{equation}
	We first show that \eqref{eqn:biquadratic} has lines everywhere locally, for which it suffices
	to consider only the non-archimedean places by Lemma \ref{lem:Chebotarev}.
	Let $\mathfrak{p}$ be a non-zero prime ideal of $k$. Then conditions $(3)$ and $(4)$,
	together with the multiplicativity of the Legendre symbol and Hensel's lemma,
	imply that one of $\alpha$,	$\beta$, or $\alpha \beta$ is a square in $k_\mathfrak{p}$. 
	Hence by \eqref{eqn:biquadratic_lines}, we see that there is a line defined over $k_\mathfrak{p}$. 
	To see that there is no
	line defined over $k$, we note that $(2)$ implies that $\beta^{d/2}$ is not a $d$th
	power in~$k$. A similar argument shows that $(\alpha\beta)^{d/2}$
	is not a $d$th power in $k$, on using $(1)$ and $(2)$.
	Moreover, $\alpha^{d/2}$ is not a $d$th power by our choice of~$\alpha$.
	Hence \eqref{eqn:biquadratic} is a counter-example to the Hasse principle for lines,
	as required.
\end{proof}

To complete the proof of Theorem \ref{thm:cohomology}, we need to construct
counter-examples when $\HH^1(k(\mu_d)/k, \mu_d) \neq 0$. By Lemma \ref{lem:pq}
there are two cases to consider, which we handle in the next two lemmata.

\begin{lemma} \label{lem:a}
	Let $d \geq 3$ and let $k$ be a number field.
	Suppose that we are in case~$(a)$ of Lemma \ref{lem:pq}.
	Then there exists a diagonal surface of degree $d$ over $k$
	which fails the Hasse principle for lines.
\end{lemma}
\begin{proof}
	By Lemma \ref{lem:pq}, we may write $d=eqp^n$
	where $p,q$ are primes with $\gcd(eq,p)=1$, and 
	$$\HH^1(k(\mu_q)/k, \mu_{p^n}(k(\mu_q)) \neq 0.$$
	In particular $q \neq 2$.
	Let $\alpha \in \OO_k$ be a representative of a non-trivial element of 
	$\HH^1(k(\mu_q)/k, \mu_{p^n}(k(\mu_q)))$ of order $p$.
	By Lemma \ref{lem:beta}, we may choose $\beta \in \ZZ$ such that
	\begin{enumerate}
	\item $\beta$ is coprime to $\alpha$.
	\item $k$ has a prime ideal $\mathfrak{b}$ such that $v_\mathfrak{b}(\beta)=1$.
	\item If $\mathfrak{p} \mid \alpha q$, then $\beta$ is a $q$th power in $k_{\mathfrak{p}}$.
	\item If $\mathfrak{p} \mid \beta$, then $\alpha$ is a $p^n$th power in $k_\mathfrak{p}$.	 
	\end{enumerate}
	We claim that the surface
	$$S: \quad  x_0^{d} - \alpha^{eq} x_1^{d} - \beta^{ep^n} x_2^{d} + \alpha^{eq} \beta^{ep^n} x_3^{d} = 0$$	
	is a counter-example to the Hasse principle for lines. To see this,
	we use the description of the Hilbert scheme of lines given in \eqref{eqn:special}. We first
	show that $S$ contains lines everywhere locally, for which it suffices
	to consider only the non-archimedean places by Lemma \ref{lem:Chebotarev}. For those primes with
	$\mathfrak{p} \mid \alpha q \beta$, locally solubility in $k_\mathfrak{p}$ follows from $(3)$ and $(4)$ above.
	For the remaining primes $\mathfrak{p}$, let $\ell$ be the rational prime below $\mathfrak{p}$.
	If $\ell \equiv 1 \bmod q$ then $\mathfrak{p}$ splits completely in $k(\mu_q)$, hence $\alpha$
	is a $p^n$th power in $k_\mathfrak{p}$ by \eqref{eqn:H^1}, thus $\alpha^{eq}$ is a $d$th power in $k_\mathfrak{p}$.
	Next, as $\beta$ is a unit in $\ZZ_\ell$, we see that if $\ell \not \equiv 1 \bmod q$ then
	$\beta$ is a $q$th power in $\QQ_\ell \subset k_\mathfrak{p}$, 
	hence $\beta^{ep^n}$ is a $d$th power in $k_\mathfrak{p}$.
	Thus $S$ has lines everywhere locally.
	
	We now show that there is no line over $k$. First assume that $\alpha^{eq}$ is a $d$th power.
	In which case, there exists $\zeta \in \mu_{eq}(k)$ such that $\zeta\alpha \in k^{*p^n}$.
	However $\zeta \in k^{*p^n}$ since $\gcd(eq,p)=1$,
	thus $\alpha \in k^{*p^n}$, contradicting our choice of~$\alpha$ (see \eqref{eqn:H^1}).
	Next $(2)$ implies that $\beta^{ep^n}$ is not a $d$th power in~$k$,
	and a similar argument (which is valid because of $(1)$)
	shows that $-\alpha^{eq}\beta^{ep^n}$ and $-\alpha^{eq}/\beta^{ep^n}$ 
	are also not $d$th powers. Thus there is no line over $k$, as required.
\end{proof}

\begin{lemma} \label{lem:b}
	Let $d \geq 3$ and let $k$ be a number field.
	Suppose that we are in case~$(b)$ of Lemma \ref{lem:pq}.
	Then there exists a diagonal surface of degree $d$ over $k$
	which fails the Hasse principle for lines.
\end{lemma}
\begin{proof}
	The proof runs along similar lines to the proof of Lemma \ref{lem:a}.
	First by Lemma~\ref{lem:pq}, we may write $d=2^ne$ with $e$ odd and
	such that $(k,2^n)$ is the special case.
	Let $\alpha \in \OO_k$ be a representative of the non-trivial element of 
	$\HH^1(k(\mu_{2^n})/k, \mu_{2^n})$.	By Lemma \ref{lem:beta}, we may choose $\beta \in \ZZ$ such that
	\begin{enumerate}
	\item $\beta$ is coprime to $\alpha$.
	\item $k$ has a prime ideal $\mathfrak{b}$ such that $v_\mathfrak{b}(\beta)=1$.
	\item If $\mathfrak{p} \mid 2\alpha$, then $\beta$ is a square in $k_{\mathfrak{p}}$.
	\item If $\mathfrak{p} \mid \beta$, then $\alpha^e$ is a $d$th power in $k_\mathfrak{p}$.	 
	\end{enumerate}
	We claim that the surface
	$$\quad  x_0^{d} - \alpha^{e} x_1^{d} - \beta^{d/2} x_2^{d} + \alpha^{e} \beta^{d/2} x_3^{d} = 0$$	
	is a counter-example to the Hasse principle for lines.
	For local solubility, those primes with
	$\mathfrak{p} \mid 2\alpha  \beta$ are taken care of using $(3)$ and $(4)$.
	For the remaining primes~$\mathfrak{p}$, let $\ell$ be the rational prime below $\mathfrak{p}$.
	If $\ell \equiv 1 \bmod 2^n$, then $\mathfrak{p}$ splits completely in $k(\mu_{2^n})$, hence $\alpha$
	is a $2^n$th power in $k_\mathfrak{p}$ by \eqref{eqn:H^1}, thus $\alpha^{e}$ is a $d$th power in $k_\mathfrak{p}$.
	Next if $\ell \not \equiv 1 \bmod 2^n$, then every element of
	$\FF_{\!\ell}$ which is a $(d/2)$th power is also a $d$th power,
	since $\gcd(\ell-1, d) \mid (d/2)$. As $\beta$ is a unit in $\ZZ_\ell$, we therefore 
	see that $\beta^{d/2}$ is a $d$th power in $\QQ_\ell \subset k_\mathfrak{p}$, as required.
	Thus $S$ has lines everywhere locally.
	
	We now show that there is no line. Suppose that $\alpha^e \in k^{*d}$. Then there exists
	$\zeta \in \mu_e(k)$ such that $\zeta \alpha \in k^{*2^n}$. However $\zeta \in k^{*2^n}$
	as $e$ is odd, thus $\alpha \in k^{*2^n}$; a contradiction.
	As in the proof of Lemma \ref{lem:a}, the remaining cases follow from $(1)$ and $(2)$.
\end{proof}

This completes the proof of Theorem \ref{thm:cohomology}. \qed

\begin{remark}
	A simple calculation using the description \eqref{eqn:general} shows that the Hilbert scheme of lines of a diagonal
	surface is a torsor under a certain non-abelian group scheme $G$, which
	fits into an exact sequence
	$$0 \to \mu_d^2 \to G \to \ZZ/3\ZZ \to 0.$$
	In particular, our results give explicit examples of non-trivial elements
	of the associated non-abelian Tate-Shafarevich set
	$$\Sha(k,G) = \ker\left( \HH^1(k,G) \longrightarrow \prod_{\mathclap{v \in \Val(k)}} \HH^1(k_v, G)\right).$$
\end{remark}

\begin{remark}
	Let $d$ be even and let $k$ be a number field with $\mu_{2d} \subset k$.
	Then a simple modification of the proof of Proposition \ref{prop:mu_d}
	shows that every diagonal surface of degree $d$ over $k$ which
	fails the Hasse principle for lines is of the form \eqref{eqn:biquadratic}.
	In particular, its splitting field has Galois group $(\ZZ/2\ZZ)^2$.
\end{remark}

\subsection{Proof of Theorem \ref{thm:odd_fail}}
Write $d = eqp^n$ with $\gcd(eq,p)=1$ and take $k = \QQ(\mu_{p^n})$.
As $q \equiv 1 \bmod p$, applying Lemma \ref{lem:non-trivial}(a) we obtain
$$\HH^1(k(\mu_d)/k, \mu_{p^n}) \neq 0.$$
Hence by Theorem~\ref{thm:cohomology}, the Hasse principle for lines can fail here.
\qed

\subsection{Proof of Theorem \ref{thm:even_fail}}
Let $k$ be totally real and write $d = 2^ne$.
First suppose that $n = 1$. Then, as $d \geq 3$ and $k$ is totally real,
the group $\Gal(k(\mu_d)/k)$ contains an element of order $2$ (e.g.~a choice of complex
conjugation for some embedding $k(\mu_d) \subset \CC$).
In particular $2 \mid [k(\mu_d): k]$, hence 
$\HH^1(k(\mu_d)/k, \mu_2) \neq~0$ by Lemma \ref{lem:non-trivial}$(a)$.
If however $n \geq 2$, then $(k,2^n)$ is the special case
(see Definition~\ref{def:special}),
hence $\HH^1(k(\mu_d)/k, \mu_{2^n}) \neq0$ by Lemma \ref{lem:non-trivial}$(b)$.

In both cases, Proposition \ref{prop:cohomology}$(a)$ implies that $\HH^1(k(\mu_d)/k, \mu_d) \neq 0$.
Thus the Hasse principle for lines can fail here by Theorem~\ref{thm:cohomology}, as required.
\qed

\end{document}